\documentclass{gtpart}     
\usepackage{amsbsy,amsfonts,latexsym,amscd,amsxtra,
amssymb}
\usepackage[all]{xy}
\usepackage{graphicx}
\usepackage{float}

\title[Action minim. properties and distances on the group of Ham. diffeos]{Action minimizing properties and distances on the group of Hamiltonian diffeomorphisms}

\author[A Sorrentino]{Alfonso Sorrentino}
\givenname{Alfonso}
\surname{Sorrentino}
\address{CEREMADE, UMR 7534 du CNRS\\
Universit\'e Paris-Dauphine\\
75775 Paris Cedex 16 \\  France. {(Current Address: Department of Pure Mathematics and Mathematical Statistics, University of Cambridge, Wilberforce Road, Cambridge CB3 0WB, United Kingdom.)}}
\email{A.Sorrentino@dpmms.cam.ac.uk}
\urladdr{}

\author[C Viterbo]{Claude Viterbo}
\givenname{Claude}
\surname{Viterbo}
\address{Centre de Math\'ematiques Laurent Schwartz \\ UMR 7640 du CNRS \\ \'Ecole Polytechnique\\ 91128 Palaiseau \\ France}
\email{viterbo@math.polytechnique.fr}
\urladdr{}

\arxivreference{1002.3915}
\arxivpassword{zzpsm}

\volumenumber{}
\issuenumber{}
\publicationyear{}
\papernumber{}
\startpage{}
\endpage{}
\doi{}
\MR{}
\Zbl{}
\received{}
\revised{}
\accepted{}
\published{}
\publishedonline{}
\proposed{}
\seconded{}
\corresponding{}
\editor{}
\version{}

\newtheorem{Teo}{Theorem}[section]
\newtheorem{Prop}[Teo]{Proposition}
\newtheorem{Cor}[Teo]{Corollary}
\newtheorem{Lem}[Teo]{Lemma}

\theoremstyle{definition}
\newtheorem{Def}[Teo]{Definition}
\newtheorem{Rem}[Teo]{Remark}

\newcommand\beqa[1]{ \begin{eqnarray} \label{#1}}
\newcommand{\eeqa}{ \end{eqnarray} }
\newcommand{\beqano}{ \begin{eqnarray*} }
\newcommand{\eeqano}{ \end{eqnarray*} }

\newcommand{\T}{ {\mathbb T}   }

\renewcommand \a {\alpha}
\newcommand \e {\varepsilon }

\renewcommand \b  {\beta}

\renewcommand \d {\delta}
\newcommand \m {\mu}

\newcommand \g {\gamma}

\renewcommand \l {\lambda}
\renewcommand \L {\Lambda}

\newcommand \cA {{\mathcal A}}

\newcommand \cH {{\mathcal H}}
\newcommand \cL {{\mathcal L}}

\newcommand \cO {{\rm O}}
\newcommand \cP {{\mathcal P}}

\newcommand \U {{\mathcal B}}

\newcommand \calM {\mathfrak{M}}

\def\ie{\hbox{ie\ }}

\newcommand \dpr {\partial}

\newcommand \rH {{\rm H}}
\newcommand \rT {{\rm T}}

\newcommand \Ham {{\mathcal Ham}}

\newcommand \Cal {{\mathcal Ca}\ell\, }
\newcommand \id {{\rm id}}
\newcommand \Osc {{\rm Osc\,}}

\begin{document}

\begin{abstract}    
In this article we prove that for a smooth fiberwise convex Hamiltonian,  the asymptotic Hofer distance from the identity gives a strict upper bound to the value at $0$ of Mather's $\beta$ function, thus providing a negative answer to a question asked by Siburg in \cite{Siburg1998}. However, we show that equality holds if one considers  the asymptotic distance defined in Viterbo \cite{Viterbo1992}.
\end{abstract}

\maketitle


\section{Introduction}

The relationship between Aubry--Mather theory  and the new tools  of symplectic topology  has attracted quite a bit of attention over the last years. 
These two approaches correspond to two different ways  of looking at  Hamiltonian systems.  While the former investigates the dynamics of the system { in the phase space}, the latter takes a more global look at the topology of the path that the Hamiltonian flow describes { in the group of Hamiltonian diffeomorphisms}. Trying to relate and combine these ``internal" and  ``external'' information is a very intriguing task. 

In this article we shall concentrate on the relation between the { action minimizing properties} of the flow of a convex Hamiltonian and its ``{asymptotic distance''} from the identity.
 
Given a {Tonelli} Hamiltonian $H$  and its corresponding Lagrangian $L$ (obtained by Legendre duality),  Aubry--Mather theory associates 
 the so-called $\beta$-{\it function} (or {\it effective Lagrangian}). Roughly speaking,  the value of such a function represents  the minimal average { Lagrangian action} needed to carry out motions with a prescribed  ``rotation vector''. The Legendre dual of the $\beta$-{function} is what is called the $\alpha$-{\it function} (or {\it effective Hamiltonian}). See section \ref{section3} for precise definitions.
  
On the other hand the corresponding Hamiltonian flow $\varphi_H^t$ determines a curve in the group of Hamiltonian diffeomorphisms. 
Recall that there are several natural metrics that can be defined on this group, in particular the so-called {\it Hofer distance} \cite{Hofer} and $\gamma$-{\it distance}\footnote{Sometimes called {\it Viterbo distance}.} \cite{Viterbo1992} (see section \ref{section2}).

Studying the connection between these objects arises quite naturally and has indeed been studied by several authors in the last years.

One question, for instance, concerns the relationship between $\beta_H(0)$, associated to a Tonelli Hamiltonian $H$, and the asymptotic Hofer distance from the identity of its time-one flow map $\varphi_H^1$. 
Observe that the definitions need to  be adjusted since $\varphi_H^1$ is not compactly supported, while the Hofer distance is only defined for compactly supported Hamiltonians. 

In \cite{Siburg1998}  Siburg proves, in the case of Hamiltonian diffeomorphisms on the cotangent disc bundle generated by a convex Hamiltonian $H$, that the asymptotic Hofer distance yields an upper bound for $\beta_H(0)$ (see Proposition \ref{ex-corollary}) and asks whether or not equality holds. 
In this paper we show that Siburg's question has a negative answer (Corollary \ref{answerSiburg}), by constructing  examples of convex Hamiltonian diffeomorphisms for which the asymptotic Hofer distance from the identity is strictly greater than the asymptotic $\g$-distance (Theorem \ref{counterex}). However, Siburg's question has a positive answer, provided  the asymptotic Hofer distance is replaced by the asymptotic $\gamma$-distance (Proposition \ref{ex-corollary}). Moreover, we extend these results to the case of general Hamiltonian diffeomorphisms generated by autonomous Tonelli Hamiltonians (Theorem \ref{Theorem1} and Corollary \ref{generalizationpropsiburg}).

Our proof  uses the theory of symplectic homogenization, for which we refer to section \ref{sectionHomog} for a short presentation and to Viterbo \cite{Viterbo2009} for more details. 
Observe that Corollary \ref{answerSiburg} is also stated in Cui \cite{Cui}. 
However, as the author  kindly confirmed to us, there  is a gap  in the proof of \cite[Proposition $7$]{Cui}. Although our proof goes along completely different lines, we are grateful to X Cui for drawing our attention to this problem. \\

\noindent{\it Acknowledgements.} We thank the anonimous referee for the highly appreciated comments and suggestions.
Alfonso Sorrentino would also like to acknowledge the supports of {\it Fondation des Sciences Math\'ematiques de Paris}, {\it Herchel-Smith foundation} and {\it Agence Nationale de la Recherche} project ``Hamilton--Jacobi et th\'eorie KAM faible''. Claude Viterbo was also supported by {\it Agence Nationale de la Recherche} projects ``Symplexe'' and ``Floer Power''.

\section{Metric structures on the group of Hamiltonian diffeomorphisms}\label{section2}
We first  define the group of Hamiltonian diffeomorphims and two metrics on this group: {\it Hofer distance} and {\it $\gamma$-distance}.\\
Let $\| \cdot \|$ denote the standard metric on the $n$-dimensional torus $\T^n \simeq {\mathbb R}^n/ {\mathbb Z}^n$ and $\omega=-d\l$ the canonical symplectic structure on $\rT^*\T^n$, where $\lambda=\sum_{j=1}^n p_{j}dq_{j}$ is the  Liouville form on $\rT^*\T^n $. 
We denote   by $\cH_0$  the set of {\it admissible} time-dependent Hamiltonians $H\in C^2(\rT^*\T^n\times \T)$, such that $H_t(q,p):=H(q,p,t)$ has compact support.
For each $H\in\cH_0$ we consider the corresponding Hamiltonian flow $\varphi_H^t$ and denote by $\varphi_H:=\varphi^1_H$ its time-one map. 
The {\it group of Hamiltonian diffeomorphisms} 
$\Ham_0(\rT^*\T^n):=\Ham_0(\rT^*\T^n,\omega)$ is the set of all Hamiltonian diffeomorphisms $\varphi: \rT^*\T^n \longrightarrow \rT^*\T^n$ that are obtained as time-one maps of elements in $\cH_0$, \ie $\varphi=\varphi_H$ for some $H\in\cH_0$.

We shall now define the Hofer and $\gamma$-distances for elements in $\Ham_0(\rT^*\T^n)$.

\subsection{The Hofer distance}

This first metric structure on the group of compactly supported Hamiltonian diffeomorphisms was defined by Hofer (see \cite{Hofer}). Consider a path in the group of compactly supported Hamiltonian diffeomorphisms, given by an admissible Hamiltonian $H$. One first defines the lenght $\ell (H)$ of this Hamiltonian path, by setting 
$$
\ell(H):= \int_{\T} {\Osc H_t}\,dt,
$$ 
where ${\Osc}H_t := \max_{\rT^*\T^n} H_t - \min_{\rT^*\T^n}  H_t$ denotes the oscillation of $H_t$.

\begin{Def}[{\bf Hofer distance}]
The  Hofer distance from the identity (or  energy) of an element $\varphi\in\Ham_0(\rT^*\T^n)$ is given by
$$
d^H({\rm id}, \varphi) := \inf \big\{\ell(H):\; H\in \cH_0\;{\rm and}\; \varphi=\varphi_H \big\}.
$$
This extends to a distance on  $\Ham_0(\rT^*\T^n)$: if $\varphi, \psi \in \Ham_0(\rT^*\T^n)$, then $d^H(\varphi, \psi):= d^H({\rm id}, \psi\circ \varphi^{-1})$.
\end{Def}

It is easy to verify that  if $H, K\in\cH_0$, then $$d^H(\varphi_{H}^1, \varphi_{K}^1) \leq \Vert H-K \Vert _{C^0}.$$

Observe that this definition only considers the flow  of a given Hamiltonian at time $t=1$. In the study of the dynamics, however,  one is interested in the long time behaviour of the system and it would be more relevant to get {global} information, such as, for instance, the {\it asymptotic Hofer distance from the identity} introduced by Bialy and Polterovich in \cite{BialyPolterovich}.

\begin{Def}[{\bf Asymptotic Hofer distance}]
Let $\varphi$ be a Hamiltonian diffeomorphism in $\Ham_0(\rT^*\T^n)$. The asymptotic Hofer distance from the identity is:
$$
d_{\infty}^H(\id,\varphi):= \lim_{k\rightarrow +\infty} \frac{d^H(\id,\varphi^k)}{k}\,.
$$
\end{Def}

\noindent It follows from the triangle inequality that the above limit exists and that $d^H_{\infty} \leq d^H$.

\subsection{The $\gamma$-distance} \label{subsec2.2}

One can also introduce another metric on $\Ham_0(\rT^*\T^n)$, commonly referred to as {\it $\gamma$-distance}. 
First of all, let us recall the following construction (see Viterbo \cite{Viterbo1992} for more details).
Let $\cL_0$ denote the set of Lagrangian submanifolds $\L$ of $\rT^*\T^n$, which are Hamiltonianly isotopic to the zero section $\cO_{\T^n}$, \ie there exists a Hamiltonian isotopy $\varphi^t$ such that $\L=\varphi^1(\cO_{\T^n})$. Consider $\L\in \cL_0$ and let $S_\L: \T^n\times \R^k \longrightarrow \R$ to be a {\it generating function quadratic at infinity} (GFQI) for $\L$ (see Viterbo \cite{Viterbo1992} for the definition). Since $\L$ is Hamiltonianly isotopic to $\cO_{\T^n}$, then $S_\L$ is unique up to some elementary operations. 
 Moreover, if we denote by $S_\L^{\lambda}:=\{(q;\xi)\in \T^n\times \R^k:\; S_\L(q;\xi)\leq \l \}$, then for sufficiently large $c\in \R$ we have that 
$\rH^*(S_\L^c,S_\L^{-c})\simeq \rH^*(\T^n)\otimes \rH^*(D^-,\dpr D^-)$, where $D^-$ is the unit disc of the negative eigenspace of the quadratic form $B$ associated to $S_\L$.  
Therefore to  each cohomology class $\a\in \rH^*(\T^n)\setminus\{0\}$, one can associate the image $\a\otimes T$ ($T$ is a chosen generator of $\rH^*(D^-,\dpr D^-)\simeq \Z$) and, by min-max methods, a critical level $c(\a,S_\L)$ (see Viterbo \cite[pages 690--693]{Viterbo1992} for more details). 

Let us now consider $\varphi\in \Ham_0(\rT^*\T^n)$. Recall that its graph $$\Gamma(\varphi)=\{(z,\varphi(z)):\; z\in\rT^*\T^n\}$$ is a Lagrangian submanifold of $\rT^*\T^n\times \overline{\rT^*\T^n}$, where $\overline{\rT^*\T^n}$ denotes ${\rT^*\T^n}$ with the sympectic form $-\omega$ (see for example Cannas da Silva \cite{Cannas}). Since $\rT^*\T^n\times \overline{\rT^*\T^n}$ is covered by $\rT^*(\Delta_{\rT^*\T^n})$, where $\Delta_{\rT^*\T^n}$
is the diagonal, we may lift $\Gamma(\varphi)$ to $\widetilde{\Gamma}(\varphi)$, which is still a Lagrangian submanifold in $\rT^*(\Delta_{\rT^*\T^n})$. Moreover, since $\varphi$ has compact support, we can compactify both
$\widetilde{\Gamma}(\varphi)$ and $\Delta_{\rT^*\T^n}$ and we obtain a Lagrangian submanifold $\overline{\Gamma}(\varphi)$ in $\rT^*(S^n\times \T^n)$. In Viterbo \cite[page 697 and page 706]{Viterbo1992} the author defined the following distance from the identity.

\begin{Def}[{\bf $\gamma$-distance}]
Let $\varphi \in \Ham_0(\rT^*\T^n)$.  
The $\gamma$-distance of $\varphi$ from the identity  is given by:
$$\gamma(\id, \varphi):= c(\mu_{\T^n}\otimes\mu_{S^n},  \overline{\Gamma}(\varphi) ) - c(1\otimes1,  \overline{\Gamma}(\varphi)).$$

In particular, this can be easily extended to a distance on the all group: if  $\varphi,\psi \in \Ham_0(\rT^*\T^n)$, then $\gamma(\varphi,\psi):= \gamma(\id, \psi\circ \varphi^{-1})$.

We also define $$c_+(\varphi):=c(\mu_{\T^n}\otimes\mu_{S^n},  \overline{\Gamma}(\varphi) )$$ and $$c_-(\varphi):=c(1\otimes1,  \overline{\Gamma}(\varphi))$$ so that $\gamma (\id, \varphi)= c_+(\varphi)-c_{-}(\varphi)$. 
\end{Def}

It is possible to show again that  $\gamma(\varphi_{H}^1, \varphi_{K}^1) \leq \Vert H-K \Vert _{C^0}$ for all $H,K\in\cH_0$, and more precisely that  $\gamma (\varphi_{H}^1,\varphi_{K}^1) \leq d^H( \varphi_{H}^1,\varphi_{K}^1)$ (see Proposition \ref{confrontodistanze}).

Analogously to what we have already seen for Hofer distance, one can introduce the  {\it asymptotic $\gamma$-distance from the identity}:
\begin{Def}[{\bf Asymptotic $\gamma$-distance}] 
Let $\varphi$ be a Hamiltonian diffeomorphism in $\Ham_0(\rT^*\T^n)$. The asymptotic $\gamma$-distance from the identity is:
$$
\g_{\infty}(\id, \varphi):= \lim_{k\rightarrow +\infty} \frac{\gamma( \id, \varphi^k)}{k}.
$$

Similarly, $$c_{\pm, \infty}(\varphi)= \lim_{k\to +\infty}\frac{ c_\pm (\varphi^k)}{k}.$$
\end{Def}

\begin{Prop}\label{confrontodistanze}
For all $\varphi \in \Ham_0(\rT^*\T^n)$, $\gamma(\id, \varphi) \leq d^H(\id,\varphi)$. In particular,
$\g_{\infty}(\id, \varphi) \leq d_{\infty}^H(\id,\varphi)$.
\end{Prop}

This is an immediate consequence of Viterbo \cite[Proposition 4.6]{Viterbo1992} and it was explicitly stated for example in Viterbo \cite[Proposition 2.15]{Viterbo2006}  or  in Humili\`ere \cite[Proposition 1.52]{Humiliere}.


 \section{Symplectic homogenization}\label{sectionHomog}
 In this section we want to provide a brief presentation of the theory of  {\it Symplectic Homogenization}, developed in Viterbo \cite{Viterbo2009}.
The main goal of this theory is to define a notion of  ``{\it homogenization}'' for Hamiltonian diffeomorphisms of $\rT^*\T^n$. More specifically, it provides an answer to the following question. Given a Hamiltonian $H(q,p,t)$, supported in a compact subset of $(q,p)\in \rT^*\T^n$ and  $1$-periodic in $t$, one would like to study whether or not the sequence of  ``rescaled'' Hamiltonians $H_k(q,p,t):=H(k\cdot q,p,kt)$ ``{\it converges}'' to some Hamiltonian $\overline{H}$, for some suitable topology.  \\
The first problem is represented by what one means by ``convergence''. A plausible interpretation would be as convergence of the time-one flows of $H_k$ to the  time-one flow of $\overline{H}$. However, since the $C^0$-convergence of the flows does not hold in general, then the chosen topology  must be rather weak. In the following we shall consider the topology induced by the $\g$-distance  defined in section \ref{subsec2.2}. Observe that the convergence in this metric does not imply any sort of pointwise or almost everywhere convergence.\\

\begin{Teo}[Theorem 4.2 in Viterbo \cite{Viterbo2009}]\label{Thm4.2}
There exists a projection operator
\beqano
\cA: C^2_0(\rT^*\T^n\times\T,\R) &\longrightarrow& C^0_0(\R^n)\\
H &\longmapsto& \overline{H}
\eeqano
such that the sequence $H_k$ $\g$-converges to $\cA(H)=\overline{H}$, \ie   the associated time-one  flows of ${H_k}$ $\g$-converge to the time-one flow \footnote{Although $\overline{H}$ is in general only continuous, one can nevertheless define its ``time-one flow'' as an element of  $\widehat{\Ham}_0(\rT^*\T^n)$, \ie the completition of $\Ham_0(\rT^*\T^n)$ with respect to $\g$. See Humili\`ere \cite{Humiliere2} for more details about this space.}
of $\overline{H}$. In particular, $\cA$ extends by $\g$-continuity to a map $\cA: C^0_0(\rT^*\T^n\times\T,\R) \longrightarrow C^0_0(\R^n)$. Moreover, the map $\cA$ satisfies the following properties: 
\begin{itemize}
\item[{\rm (1)}] It is monotone, \ie if $H_1\leq H_2$, then $\cA(H_1)\leq \cA(H_2)$.
\item[{\rm (2)}] It is invariant by Hamiltonian symplectomorphisms, \ie $\cA(H\circ \psi) = \cA(H)$ for all $\psi \in \Ham(\rT^*\T^n)$.
\item[{\rm (3)}] We have $\cA(-H)=\cA(H)$.
\item[{\rm (4)}] The map $\cA$ extends to characteristic functions of subsets, hence induces a map (still denoted by $\cA$) between $\cP(\rT^*\T^n)$, the set of subsets of $\rT^*\T^n$, to $\cP(\R^n)$, the set of subsets of $\R^n$. This map is bounded by the symplectic shape of Sikorav {\rm \cite{Sikorav}}, \ie
$$
{\rm shape}(U) = \{p_0\in \R^n:\; \exists \psi\in \Ham_0(\rT^*\T^n),\; \psi(\T^n\times \{p_0\}) \subset U\} \subset \cA(U).
$$
\item[{\rm (5)}] If $\L$ is a Lagrangian submanifold Hamiltonianly isotopic to $\L_{p_0}=\T^n\times\{p_0\}$ and $H\big|_{\Lambda} \geq h$ {\rm(}resp. $\leq h${\rm)}, then $\cA(H)(p_0)\geq h$ {\rm(}resp. $\leq h${\rm)}.
\item[{\rm (6)}] We have 
\beqano
\lim_{k\rightarrow +\infty} \frac{1}{k}c_+(\varphi_H^k) &=& \sup_{p\in \R^n} \overline{H}(p)\\
\lim_{k\rightarrow -\infty} \frac{1}{k}c_-(\varphi_H^k) &=& \inf_{p\in \R^n} \overline{H}(p)\,.
\eeqano
\item[{\rm (7)}] Given any measure $\m$ on $\R^n$, the map
$$
\zeta(H):= \int_{\R^n} \cA(H)(p)\,d\m(p)
$$
is a symplectic quasi-state {\rm(}cf {\rm Entov and Polterovich \cite{Entov-Polterovich}}{\rm)}. In particular we have $\cA(H+K)=\cA(H)+\cA(K)$, whenever $H$ and $K$ Poisson-commute {\rm(}\ie $\{H,K\}=0${\rm)}.
\end{itemize}
\end{Teo}

We refer to Viterbo \cite[Section 6]{Viterbo2009} for a proof of this theorem. Observe that property (6), which will play a crucial role in our proofs, can be deduced from the fact that 
$c_{\pm}$ are continuous with respect to the $\g$-topology (see Viterbo \cite{Viterbo1992}) and that
$\frac{1}{k}c_{\pm}(\varphi_{H}^k)= c_{\pm}(\rho_k^{-1}\varphi_{H}^k\rho_k),$
where $\rho_k(q,p):=(k\cdot q,p)$.

\begin{Rem}\label{Rem4.3}
\begin{itemize}
\item[{\rm (i)}] As a result of (5), if $u$ is a smooth subsolution of the stationary Hamilton--Jacobi equation, that is 
$H(q,p+du(q))\leq h$, then $\overline{H}(p)\leq h$.
Similarly, if $u$ is a smooth supersolution, that is 
$H(q,p+du(q))\geq h$, then $\overline{H}(p)\geq h$.
\item[{\rm (ii)}] From (5) we also get the following statement. Let
\beqano
E_c^+ &=& \{p_0\in \R^n:\; \exists\, \L\; \mbox{Lag. subman. Hamilton. isotopic to}\; \L_{p_0},\; H\big|_{\L} \geq c \}\\
E_c^- &=& \{p_0\in \R^n:\; \exists\, \L\; \mbox{Lag. subman. Hamilton. isotopic to}\; \L_{p_0},\; H\big|_{\L} \leq c \}\,.
\eeqano
If $p\in \overline{E^+_c}\cap \overline{E^-_c}$, then $\overline{H}(p)=c$.
\end{itemize}
\end{Rem}

So far we have considered compactly supported Hamiltonians. Actually the whole theory can be also extended to the non-compact case,  but one needs to impose some conditions on the ``growth'' of the Hamiltonian. We shall say that a Hamiltonian $H$ is {\it coercive} if 
$$
\lim_{\|p\|\rightarrow + \infty} H(q,p, t)=+\infty.
$$

Let us  describe the autonomous case, which we shall use in the sequel. Given $H: \rT^*\T^n \longrightarrow \R$ a coercive Hamiltonian,
the basic idea consists in considering a truncation of the Hamiltonian $H_A(q,p):=\chi_A(\|p\|)H(q,p)$, where $\chi_A:\R \longrightarrow \R$ is supported on $[-2A,2A]$ and $\chi_A(s)\equiv 1$ on $[-A,A]$.
For any $A>0$, this new Hamiltonian satisfies the hypotheses of Theorem \ref{Thm4.2} and one can therefore define its homogenization $\overline{H}_A$.
We then define $\overline{H}:= \lim_{A\rightarrow +\infty} \overline{H}_A$. It is possible to check (see Viterbo \cite[Section 9.1]{Viterbo2009}) that this function is well defined up to a constant. Hence,

\begin{Prop}[{Proposition 9.3 in Viterbo \cite{Viterbo2009}}]\label{prop9.3} The map $\cA$ extends to a map defined on the set of autonomous coercive Hamiltonians.
\end{Prop}

For non-autonomous coercive Hamiltonians $H(q,p,t)$, one can reduce to the autonomous case by considering the new Hamiltonian on $\rT^*\T^{n+1}$ given by $$K(q,p,t, \tau):= \tau + H(q,p,t).$$ We refer the reader to Viterbo \cite[Section 9.3]{Viterbo2009} for more details.


\section{Action minimizing properties of convex Hamiltonians}\label{section3}
 
 In this section we want to recall some notions of {Mather's theory of minimal action}. In order to do this, we need to restrict our analysis to a special class of Hamiltonians $H: \rT^*\T^n \times \T \longrightarrow \R$, which are $C^2$, strictly convex and superlinear in the fibers and have complete flows. These Hamiltonians, which  are also called {\it Tonelli Hamiltonians}, play an important role in the study of classical mechanics and provide the setting in which {\it Mather's theory} and {\it Fathi's Weak KAM theory} have been developed (see for instance Mather \cite{Mather91}, Fathi \cite{Fathibook} or Sorrentino \cite{SorLecturenotes}). 

 Let $H$ be a Tonelli Hamiltonian 
 and consider the associated Lagrangian  $L:\rT \T^n \times \T \longrightarrow \R$, which is  
 defined by Legendre duality  by the formula:
 $$L(q,v,t)=\sup \{\langle p, v \rangle -H(q,p,t) \mid p\in {\mathbb R}^n\}.         $$
Recall that the  associated Euler--Lagrange flow $\varphi_L^t$ is obtained as the solution of the equation  $\frac{d}{dt} \frac{\dpr L}{\dpr v}(q,v,t) = \frac{\dpr L}{\dpr q}(q,v,t)$ and it is conjugated, via the Legendre transform
$(q,v,t) \longmapsto \big(q,\frac{\dpr L}{\dpr v}(q,v,t),t\big)$,  to the Hamiltonian flow $\varphi_H^t$. 
Let  $\m$ be a probability measure on ${\rm T}\T^n \times \T$, which is invariant under the Euler--Lagrange flow (\ie $(\varphi_{L}^t)_{*}\mu=\mu$). We define its {\it average action} as
$$ A_L(\m) := \int_{{\rm T}\T^n \times \T} L(q,v,t)\,d\m\,. $$
Let us denote by $\calM(L)$ the space of invariant probability measures on 
${\rm T}\T^n \times \T$, with finite average action.
Given any $\m \in \calM(L)$ we can define its {\it rotation vector} or {\it Schwartzman's asymptotic cycle} as the unique $\rho(\m) \in \rH_1(\T^n;\R)$ that satisfies
$$
\int_{\rT\T^n\times \T} \eta(q,t)\cdot (v,1)\,d\mu = \langle \rho(\mu), [\eta]_{\T^n}\rangle + [\eta]_{\T}
$$
for any closed $1$-form $\eta$ on $\T^n\times \T$, where $[\eta]=([\eta]_{\T^n}, [\eta]_{\T}) \in 
\rH_1(\T^n\times\T;\R)\simeq \rH_1(\T^n;\R)\times \R$ is the de-Rham cohomology of $\eta$.
It is possible to show (see Mather \cite{Mather91}) that the map $\rho: \calM(L) \longrightarrow {\rm H}_1(\T^n;\R)$ is surjective and hence there exist invariant probability measures for each rotation vector.
Let us consider the minimal value of the average action $A_L$ over the set of 
probability measures with a given rotation vector. This minimum exists because of the lower semicontinuity of the action functional on the set $\calM(L)$ for the weak-$*$ topology (see Mather \cite{Mather91}):
\begin{eqnarray}
\b_H: {\rm H}_1(\T^n;\R) &\longrightarrow& \R \nonumber\\
h &\longmapsto& \min_{\m\in\calM(L):\,\rho(\m)=h} A_L(\m)\,. \nonumber
\end{eqnarray}
This function $\beta_H$ is what is generally known as {\it  
$\beta$-function} or {\it effective Lagrangian}.
A measure $\m \in \calM(L)$ realizing such a minimum amongst all invariant probability measures with the same rotation vector, \ie $A_L(\m) = 
\b(\rho(\m))$, is called an {\it action minimizing} {\it measure with rotation vector} $\rho(\mu)$. 
The $\beta$-function is convex  and therefore one can consider its {\it conjugate} 
function (given by Fenchel duality)
$ \a_H: {\rm H}^1(\T^n;\R) \longrightarrow \R $ defined by
\beqano
\a_H(c) &:=& \max_{h\in {\rm H}_1(\T^n;\R)} \left(\langle c,h \rangle - \b_H(h) \right).
\eeqano
This  function  is generally called {\it $\a$-function} or {\it effective Hamiltonian}. See Mather \cite{Mather91} for more details. \\

It turns out that $\a_H$  coincides with $\overline{H}$, \ie the symplectic homogenization of $H$  introduced in section \ref{sectionHomog}.
More precisely, 

\begin{Prop}[{Proposition 10.3 in Viterbo \cite{Viterbo2009}}]
If $H:\rT^*\T^n\times \T \longrightarrow \R$ is a Tonelli Hamiltonian, then $\a_H = \overline{H}$
\end{Prop}

\begin{Rem}
\begin{itemize}
\item[{\rm (i)}] Observe that a Tonelli Hamiltonian is of course coercive (since it is superlinear), therefore ${\overline{H}}$ is defined as in Proposition \ref{prop9.3}. Moreover, in the autonomous case we simply have $\overline H(p)=\alpha_{H}(p)$, while for non-autonomous Hamiltonians, we first reduce to the autonomous case by setting $K(q,p,t,\tau)=\tau + H(q,p,t)$ and then observe that  $\overline K (p,\tau)$ is well defined and equal to $\tau + \overline H(p)$ for some function $\overline H$. For this function, we have again $\overline H(p)=\alpha_{H}(p)$. See Viterbo \cite[Section 10.1]{Viterbo2009}.
\item[{\rm (ii)}] It follows from this result and Theorem \ref{Thm4.2} (2) that Mather's $\a$ function is invariant under Hamiltonian symplectomorphisms. This property had already been proved for general symplectomorphisms by Patrick Bernard \cite{Bernard} (see also Sorrentino \cite[Section 4.A]{SorLecturenotes}).
\end{itemize}
\end{Rem}


\section{Main results}\label{sec-mainresults}

In this section we want to study the connection between the  Hofer and $\gamma$-distance on one hand, and Mather's theory of minimal action on the other hand. We shall then state our main results. 
Let us start by observing that Tonelli Hamiltonians clearly do not  belong to $\cH_0$, since they lack compact support, and hence their time-one maps are not element of $\Ham_0(\rT^*\T^n)$. Therefore, we need to restrict them to compact subsets of $\rT^*\T^n$ and, as done by Siburg in 
\cite{Siburg1998},  consider ``nice''  compactly supported extensions, such that the Hofer and $\gamma$- metrics are independent of the choice of the extension. 
In the autonomous case this can be achieved by using the conservation of energy to obtain well-defined ``truncated'' flows  and then smooth them out. We shall then see how the same proof extends to the setting considered in Siburg \cite{Siburg1998}.

Let us consider $H(q,p)$ an autonomous Tonelli Hamiltonian on $\rT^*\T^n$. 
 Then,
 for each $r$ and each sufficiently small $\e>0$, let us consider functions $f_{r, \varepsilon }$ such that: 
\begin{itemize}
\item  $f_{r, \varepsilon }(s)=s$ for $  s  \leq r$; 
\item $f_{r, \varepsilon }(s)\leq r+ \varepsilon$ for all $s\in \R $ and $f_{r, \varepsilon }(s)\equiv r$ for $s\geq r+\e$;
\item $  \vert f'_{r, \varepsilon }(s) \vert \leq 1$ for all $s\in \R$ (this assumption will only be used in section \ref{proofcounterex} to define the Calabi invariant).
\end{itemize}
Now, we can define new Hamiltonians, given by $H_{r, \varepsilon }=f_{ r, \varepsilon }(H)$. If we denote by
$S_r:= \{(q,p)\in \rT^*\T^n:\; H(q,p)\leq r\}$, then our new Hamiltonians will satisfy the following conditions:
	\begin{itemize}
	\item $H_{r,\e}-r$ is supported  in  $S_{r+\e}$;
	\item $H_{r,\e} $ coincides with $H$ on $S_r$;
	\item $H_{r,\e}$ is bounded everywhere by  $ r+\varepsilon $ and satisfies the condition $ \| d_pH_{r,\varepsilon }\| \leq C$, where $C= \sup_{S_{r+\e}}\Vert d_pH \Vert$.
	\end{itemize}

We shall denote the set of all these possible extensions by $\cH_{r,\e}(H)$. We can now define Hofer and $\g$-distances from the identity: 
\begin{itemize}
\item[-] {\bf Hofer distance.}  We set 
$$
		d^H({\rm id}, \varphi; S_{r}) :=\lim_{ \varepsilon \to 0 } d^H(\id , \varphi_{{H}_{r, \varepsilon }}).
$$ 
		
		The above limit is well defined since if $H_{r,\e}$ and $K_{r,\e}$ are two different extensions of $H$ in $\cH_{r,\e}(H)$, then
		$\Vert H_{r, \varepsilon }- K_{r, \varepsilon } \Vert _{C^0} \leq 2\varepsilon $.\\

		Note that $d^H(\id , \varphi ; S_{r})$  only   depends on $S_{r}$ and $\varphi$,  and not on $H$. We could also  take a non-autonomous Hamiltonian, possibly non-convex, generating $\varphi$, provided it coincides with $H$ near $S_{r}$. Indeed, the time one flow of  $\varphi^t_{{H}_{r, \varepsilon }}=\varphi_{H}^{t f'_{r, \varepsilon }(H)}$ is determined by the knowledge of $\varphi$ inside $S_{r}$ and of $\varphi^t_{H}$ in   $S_{r+ \varepsilon }\setminus S_{r}$. But the latter is determined by the hypersurfaces $\partial S_{s}$ for $r \leq s \leq r+ \varepsilon $.  In any case we have the following lower bound: if $\varphi \neq \id$ there is a ball in $S_{r}$ such that $\varphi (B)\cap B=\emptyset$ and then  $d^H({\rm id}, \varphi; S_{r}) \geq c(B)>0$, where $c(B)$ is the Ekeland--Hofer capacity of $B$ (cf \cite{Ekeland-Hofer}).
		
\item[-] {\bf $\g$-distance.}  It can be defined as:
	  $$\gamma(\id, \varphi;S_{r}):=\lim_{ \varepsilon \to 0} \gamma( \id, \varphi_{H_{ r,\varepsilon }}).$$
            Also this limit is well defined, for the same reasons as above. In fact,             
            if $H_{r,\e}$ and $K_{r,\e}$ are two different extensions of $H$ in $\cH_{r,\e}(H)$, then
	   $$\g(\varphi^1_{H_{r,\e}},\varphi^1_{K_{r,\e}}) \leq \Vert H_{r, \varepsilon }- K_{r, \varepsilon } \Vert _{C^0} \leq 2\varepsilon.$$
	   Similarly, $c_{\pm}(\varphi;S_r):= \lim_{\varepsilon \to 0} c_{\pm}(\varphi_{H_{ r,\varepsilon }})$.
\end{itemize}

The same argument as before shows that $\gamma (\id, \varphi;S_{r})$ depends only on $S_{r}$ and $\varphi$, and we have the same lower bound as for $d^H$. 
Analogously one can also define the associated asymptotic quantities $d^H_{\infty}$, $\g_{\infty}$ and
$c_{\pm,\infty}$.\\

We can now state our first result.

\begin{Teo}\label{Theorem1}
Let $\varphi$ be the flow of an autonomous Tonelli  Hamiltonian $H:\rT^*\T^n  \longrightarrow \R$ and let $S_{r}\equiv 
 \{(q,p)\in\rT^*\T^n:\; H(q,p)\leq r \} $. Then, for each $r>\inf_{p \in {\mathbb R}^n} \alpha_{H}(p)$: 
 $$
\g_{\infty}(id,\varphi; S_{r}) = r+\b_H(0).
$$
More precisely,
$c_{-,\infty}(\varphi ; S_{r})= \inf_{p \in {\mathbb R}} {\overline H}(p)=-\beta_{H}(0)$ and $c_{+,\infty}(\varphi ; S_{r})=r.$
\end{Teo}

Observe now that using Proposition \ref{confrontodistanze}, we obtain the following result.

\begin{Cor}\label{generalizationpropsiburg}
Let $\varphi$ be a Hamiltonian diffeomorphism generated by  a Tonelli Hamiltonian
 $H:\rT^*\T^n \longrightarrow \R$. Then, for each $r>\inf_{p \in {\mathbb R}^n} \alpha_{H}(p)$: 
 $$
d^H_{\infty}(\id, \varphi; S_{r})  \geq r+\b_H(0).
$$
\end{Cor}

The method used to prove Theorem  \ref{Theorem1} allows us to provide a new proof of Siburg's result in 
\cite{Siburg1998} (see also Iturriaga and Sanchez \cite{IturriagaSanchez} for a generalization to general cotangent bundles). 
 Let  $B^*\T^n$ denote the unit ball cotangent bundle of $\T^n$, \ie $B^*\T^n:=\{(x,p):\; \|p\|\leq 1 \}$. Siburg  considered the set of admissible Hamiltonians
$\cH_S$ consisting of all smooth convex Hamiltonians $H: B^*\T^n \times \T \longrightarrow \R$ that satisfy the following two conditions:
\begin{itemize}
\item[{\small S1})] $H$ vanishes on the boundary of $B^*\T^n$, \ie $H(q,p,t)=0$ if $\|p\|=1$;
\item[{\small S2})] $H$ admits a smooth extension $K_H: \rT^*\T^n\times \T \longrightarrow \R$ that is of Tonelli type and depends only on $t$ and $\|p\|^2$ outside $B^*\T^n\times \T$. 
\end{itemize}
Then, he defined the following group:
$$
\Ham_S(B^*\T^n):= \{\varphi: B^*\T^n \longrightarrow B^*\T^n \; \big| \; \varphi=\varphi_H^1\; \mbox{for some}\; H\in\cH_S   \}.
$$

Observe that we cannot apply directly Theorem \ref{Theorem1} in this setting, since these Hamiltonian diffeomorphisms are not necessarily generated by autonomous Hamiltonians. However, due to the special form of the Hamiltonians, the above results remain true.\\

\begin{Prop}\label{ex-corollary}
Let $\varphi \in \Ham_S(B^*\T^n)$. Then:
$$
d^H_{\infty}(\id, \varphi) \geq \gamma_{\infty}(\id, \varphi) = \b_H(0).
$$
More precisely,
$c_{-,\infty}(\varphi)= \inf_{p \in {\mathbb R}} {\overline H}(p)=-\beta_{H}(0)$ and $c_{+,\infty}(\varphi)=0.$\\
\end{Prop}

\begin{Rem}
({\it i}) The fact that $\beta_{H}(0)$ is independent of the extension $K_H$, has been proven in Siburg \cite[Lemma 4.1]{Siburg1998}. Therefore, we define $\beta_H(0):=\beta_{K_H}(0)$. Observe that the independence could be also deduced from the above proposition and the fact that the $\g$-distance depends only on $\varphi$.\\
\noindent ({\it ii})  The inequality \beqa{Siburginequality} d^H_{\infty}(\id, \varphi) \geq \b_{H}(0) \eeqa  is due to Siburg \cite[Theorem 5.1]{Siburg1998}. 
\end{Rem}

This lower bound (\ref{Siburginequality}) induced Siburg to ask the following question:\\

\noindent
{\bf Question}  {\cite[page 94]{Siburg1998}:} does equality hold in (\ref{Siburginequality})?\\

We shall  construct examples  of convex Hamiltonian diffeomorphisms for which the asymptotic Hofer distance from the identity is strictly greater than the asymptotic $\g$-distance.

\begin{Teo}\label{counterex}
There exists $\varphi \in \Ham_S(B^*\T^n)$, such that
$$
\gamma_{\infty}(\id, \varphi) <  d^H_{\infty}(\id, \varphi)\,.
$$
\end{Teo}

An easy consequence of Theorem \ref{counterex}  and Proposition \ref{ex-corollary} is that the above  question has a negative answer.

\begin{Cor}\label{answerSiburg}
There exists $\varphi \in \Ham_S(B^*\T^n)$ generated by  convex Hamiltonian $H$, such that
$$
d^H_{\infty}(\id, \varphi) > \beta_H(0)\,.
$$
\end{Cor}

\subsection{Proof of Theorem \ref{Theorem1}}
	
\begin{proof} [Proof of Theorem 1]	
Using Theorem \ref{Thm4.2} (6), we obtain
\beqano c_{-,\infty}(\varphi_{H}; S_{r}) &=&  \lim_{\e\to 0} c_{-,\infty}(\varphi^1_{H_{r,\e}}) =\\
&=& \lim_{\e\to 0} \inf_{p \in \R^n} \overline{H}_{r,\e}(p).
\eeqano

Similarly,
\beqano
 c_{+,\infty}(\varphi_{H}; S_{r}) &=&  \lim_{\e\to 0} c_{+,\infty}(\varphi^1_{H_{r,\e}}) =\\
&=& \lim_{\e\to 0} \sup_{p \in \R^n} \overline{H}_{r,\e}(p).
\eeqano

Let $H_r:= \min \{H,r\}$. Observe that since $ \Vert H_{r, \varepsilon }- H_{r} \Vert_{C^0} \leq \varepsilon$, then $
\lim_{ \varepsilon \to 0}  {\overline H}_{r, \varepsilon } =  \overline{H}_r$ uniformly. We want to prove that  $\inf_{p\in {\mathbb R}^n} {\overline H}_{r}(p)=\inf_{p\in {\mathbb R} ^n} {\overline H}(p)$. Clearly we have $\inf_{p}{\overline H}_{r}(p)\leq \inf_{p}{\overline H}(p)$, since $H_{r}\leq H$.  We thus have to prove that
$\inf_{p}{\overline H}_{r}(p)\geq  \inf_{p}{\overline H}(p)$.

We shall need the following lemmata.

\begin{Lem}\label{Lemma1}
Let $f_{r, \varepsilon }$ be a function such that $f_{r, \varepsilon }(x)=x$ for $ \vert x \vert \leq r$. Let us consider a Hamiltonian $H:\rT^*\T^n\longrightarrow \R$ and set $H_{r, \varepsilon }=f_{r, \varepsilon }(H)$.
If there exists a Lagrangian submanifold $\L_p$  of cohomology class $p$, such that $\max_{\L_p}H \leq r$, then  $\overline H_{r, \varepsilon }(p)=\overline H(p)$.
\end{Lem}

\begin{proof}
Since clearly the Poisson bracket of $H_{r, \varepsilon }$ and $H$ vanishes, we have
$$\overline{H_{r, \varepsilon }-H}=\overline{H}_{r, \varepsilon } - \overline{H}$$ (see Theorem \ref{Thm4.2} (7)). It is thus enough to prove that $\overline{H_{r, \varepsilon }-H}(p)=0$. In other words, let $K$ be a Hamiltonian vanishing on $\L_p$, we must show that  $\overline K(p)=0$. But this follows immediately from Remark \ref{Rem4.3} (ii).
\end{proof}

\begin{Lem} 	
Let $f$ be a function in $C^0( {\mathbb R} ,{\mathbb R} )$ and $H:\rT^*\T^n\longrightarrow \R$ a convex superlinear Hamiltonian. Then, 
$\overline{f(H)}=f({\overline H})$.	
\end{Lem} 

\begin{proof} 
Let us first consider the case in which $f$ is  smooth, convex and strictly increasing. Then, $f(H)$ is smooth, convex and superlinear and, using the characterization of the $\a$-function given in Contreras, Iturriaga and Paternain \cite{Cont-Itu-Pat} and Proposition \ref{prop9.3}, we obtain:
$$\overline {f(H)}(p)=\inf_{u \in C^{\infty}(\T^n)} \sup_{q\in \T^n}f(H(q,p+du(q))).$$ 
But  this clearly coincides with
$f( \inf_{u \in C^{\infty}(\T^n)} \sup_{q\in \T^n}H(q,p+du(q))=f(\overline H(p))$. Now because $f(H)$ and $g(H)$ commute for any $f$ and $g$, we have $ \overline{f(H)-g(H)}=\overline {f(H)}-\overline {g(H)}$ (see Theorem \ref{Thm4.2} (7)). Since the set of differences of convex increasing functions contains the set of $C^{2}$ functions, and these form a dense subset among continuous functions\footnote{Indeed we may also write $f(x)= (f(x)+cx^2)-cx^2$, and   for $c$ sufficiently large both terms are convex.}, this concludes our proof. 
\end{proof} 

As a result, we have that $\overline{H_{r}}=\overline{\min \{H,r\}}= \min \{\overline H , r\}$, hence  $\inf_{p}\overline H_{r}= \inf_{p}\{\overline H, r\}= \inf_{p} \overline H (p)$, provided $r> \inf \overline H(p)$.  And by the same argument, $\sup_{p}\overline H_{r}= \sup_{p}\min\{\overline H, r\}=r$ since $\lim_{ \Vert p \Vert \to \infty} \overline H (p)= +\infty$. This concludes the proof of Theorem \ref{Theorem1}. \\
\end{proof}

\subsection{Proof of Proposition \ref{ex-corollary}}\label{subsecprop2}

Before entering into the details of the proof, let us observe that also in this case
$\varphi \in \Ham_S(B^*\T^n)$ is not necessarily compactly supported, therefore we need to define compactly supported extensions. Let us denote $B_r^*\T^n:=\{(q,p):\; \|p\|\leq r \}$. 
If $H$ is a generating Hamiltonian for $\varphi$, for each $\e$ sufficiently small we can consider new compactly supported Hamiltonians $H_{\e}= f_{ 0, \varepsilon}(K_{H})$ supported in a neighborhood of $B^*\T^n\times \T$. The Hamiltonian $H_{ \varepsilon }$ has the following properties (see section \ref{sec-mainresults} for the properties of $f_{0,\e}$):
\begin{enumerate} 
\item [(1)] it coincides with $H$ in $B^*\T^n \times \T$;
\item [(2)] it is  non-negative and bounded by $ \varepsilon $ outside  $B^*\T^n \times \T$ and satisfies $ \| d_pH_{ \varepsilon }\| \leq C$ outside $B^*\T^n \times \T$,  for some constant $C$ depending  only on $H$;
\item [(3)] it  depends only on $\| p \|^2$ and $t $ outside $B^*\T^n \times \T$.
\end{enumerate} 
We denote  the set of all these Hamiltonians by $\cH_{ \varepsilon }(H)$.  
As before, we set 
\beqano
d^H({\rm id}, \varphi) &:=& \lim_{ \varepsilon \to 0 } d^H(\id , \varphi_{{H}_{\varepsilon }})\\
 \g({\rm id}, \varphi) &:=& \lim_{ \varepsilon \to 0 } \g (\id , \varphi_{{H}_{\varepsilon }})\\
c_{\pm}(\varphi) &:=& \lim_{ \varepsilon \to 0 } c_{\pm} (\varphi_{{H}_{\varepsilon }}).
\eeqano
For the same reasons as before,  these definitions are independent of the chosen $H_{\e}$ and one can define analogously the associated asymptotic quantities.

 \begin{proof} [Proof of Proposition \ref{ex-corollary}]
 Note first that the left-hand side inequality is just Proposition \ref{confrontodistanze}.  
Let $\varphi \in \Ham_{S}(B^*\T^n)$ be generated by a convex Hamiltonian $H$ and let $K_H$ denote its smooth extension to $\rT^*\T^n \times \T$, given by condition (S2) in the definition of $\cH_S$. Then,
\beqa{espressionebeta}
\beta_H(0) &:= &\b_{K_H}(0) = \sup_{c\in\rH^1(\T^n;\R)} \big(\langle c,0\rangle - \a_{K_H}(c) \big) = - \inf_{c\in\rH^1(\T^n;\R)} \a_{K_H}(c) = \nonumber\\
&=& - \inf_{p\in\R^n} \overline{K}_H(p),
\eeqa
where in the last equality we used that $\a_{K_H}$ coincides with the symplectic homogenization $\overline {K}_H$.

Let $\U_1$ denote the closed unit ball in $\R^n$ with the standard norm. First, we want to prove that $\overline{K}_H \equiv 0$ on $\dpr \U_1$. Then, using the convexity of $\U_1$, it follows easily that $\overline{K}_H \leq 0$ in $\U_1$ and that:
$$
\max_{\U_1} \overline{K}_H=0 \qquad {\rm and} \qquad \min_{\U_1} \overline{K}_H=\min_{p \in \R^n} \overline{K}_H.
$$

\noindent In order to prove that $\overline{K}_H$ vanishes on $\dpr \U_1$, observe that for each $p_0 \in \dpr\U_1$, ${K_H}$ vanishes on the Lagrangian submanifold $\L_{p_0}:=\T^n \times \{p_0\}$. 
The claim then follows from Theorem \ref{Thm4.2} (5).

The proof of Proposition \ref{ex-corollary} can now be obtained applying Theorem \ref{Thm4.2} (6).

In fact, 
\beqa{formulaprop2}
\gamma_{\infty}(\id,\varphi) &=& \lim_{\e \to 0} \gamma_{\infty}(\id,\varphi_{H_\e}) = \nonumber\\
&=&\lim_{\e\rightarrow 0} c_{+,\infty}(\varphi_{H_\e})  -  \lim_{\e\rightarrow 0} c_{-,\infty}(\varphi_{H_\e}) =\nonumber\\ 
&=& \lim_{\e\to 0} \sup_{p \in \R^n} \overline{H}_{\e}(p) - \lim_{\e\to 0} \inf_{p \in\R^n} \overline{H}_{\e}(p).
\eeqa

Using the fact that $K_H$ and $H_\e$ coincide on $B^*\T^n$ and proceeding as in Lemma \ref{Lemma1}, one can deduce that $\overline{H}_{\e}(p)=\overline{K}_{H}(p)$ for each $p\in\U_1$.
Moreover, since $H_\e$ is bounded by $\e$, we obtain:
\beqano
0&=&\sup_{p\in \U_1} \overline{K}_{H}(p)=\sup_{p\in\U_1} \overline{H}_{\e}(p) \leq
\lim_{\e\to 0} \sup_{p \in\R^n} \overline{H}_{\e}(p) \leq \lim_{\e\to 0} \e =  0.
\eeqano

Similarly for the infimum of $\overline {H}_{ \varepsilon }$, we have

\beqano
0 \geq \inf_{p\in \U_1} \overline{K}_{H}(p) =
\inf_{p\in\U_1} \overline{H}_{\e}(p) = \lim_{\e\to 0}\inf_{p\in \U_1} \overline{H}_{\e}(p),  \eeqano
using for the last equality the fact that the left-hand side is independent of $ \varepsilon $. 
Now we use the fact that $H_{ \varepsilon }\geq 0$ on ${\mathbb T}^n\times \{p\}$ for any $p \in {\mathbb R} ^n\setminus \U_{1}$ and Theorem \ref{Thm4.2} (5),
to conclude that $\overline H_{ \varepsilon }(p) \geq $ for $p \in {\mathbb R} ^n\setminus \U_{1}$, and hence $\inf_{p \in {\mathbb R}^n} \overline H_{ \varepsilon }(p)= \inf_{p \in \U_{1}} \overline H_{ \varepsilon }(p)$.

Therefore, substituting in (\ref{formulaprop2}) and using (\ref{espressionebeta}) we can conclude:

\beqano
\gamma_{\infty}(\id,\varphi) &=& \lim_{\e\to 0} \sup_{\R^n} \overline{H}_{\e}(p) - \lim_{\e\to 0} \inf_{\R^n} \overline{H}_{\e}(p) =\\
&=& - \inf_{p\in\U_1} \overline{K}_H(p) = \beta_H(0).
\eeqano

 \end{proof}

\subsection{Proof of Theorem \ref{counterex}}\label{proofcounterex}

In this section we shall construct   examples of   $\varphi \in \Ham_{S}(B^*\T^n)$ generated by a convex Hamiltonian $H$, such that
$\gamma_{\infty}(\id,\varphi) <  d^H_{\infty}(\id, \varphi)\,.$

\begin{proof} [Proof of Theorem \ref{counterex}]
The basic observation is that in the compactly supported case the Hofer distance can be bounded from below in terms of the so-called {\it Calabi invariant} (see Calabi \cite{Calabi} and Banyaga \cite{Banyaga}):
$$\Cal(\varphi):= \int_0^1\int_{\rT^*\T^n} H(q,p,t)\,\omega^n\,dt.$$
This invariant only depends on $\varphi$ and not on the path defined by $H$. 
Indeed, let us consider the Liouville form $\lambda=\sum_{j=1}^n p_{j}dq_{j}$ on $\rT^*\T^n $. Then $\varphi^*\lambda-\lambda= df_{\varphi}$, and in the compactly supported case, an easy computation (see Banyaga \cite{Banyaga} for instance, or Calabi \cite{Calabi} where this is used as the original definition) shows that  
\begin{equation} \label{Cal}\Cal (\varphi)=\frac{1}{n+1}\int_{\rT^*\T ^n} f_{\varphi}\, \omega^n. \end{equation}  
Let us adapt this to the case of $\Ham_S (B^*\T^n)$. In our situation we set again 
$$\Cal(\varphi)=\int_0^1\int_{B^*\T^n} H(q,p,t)\,\omega^n\,dt.$$
However it is not obvious that $\Cal (\varphi)$ as defined only depends on $\varphi$, and not on the choice of the path defined by $H$. 
To prove this, notice that if $H_{ \varepsilon }$ is a compactly supported extension of $H$ (as defined in section \ref{subsecprop2}), we have 
\beqa{limitcalabi}\lim_{ \varepsilon \to 0}\, \Cal (\varphi_{H_{ \varepsilon }})= \Cal (\varphi).\eeqa
Indeed, let us denote by $B_{\rho ( \varepsilon ) }^*\T^n$ a disc bundle of radius $\rho ( \varepsilon )$ that contains the support of $H_{ \varepsilon }- \varepsilon $ (\ie contains $K_{H}^{-1}( \varepsilon )$), and notice that since $d_{p} K_{H}(q,p,t)\cdot p $ is non-zero near $\partial B^*\T^n$ we have that  $\rho ( \varepsilon )-1 = O( \varepsilon )$. Applying formula (\ref{Cal})  to $\varphi_{H_{ \varepsilon }}$ instead of $\varphi$, we get
 \begin{equation} \label{Cal2} \Cal (\varphi_{H_{ \varepsilon }})=\frac{1}{n+1}\int_{B^*\T^n} f_{\varphi}+ \frac{1}{n+1}\int_{B_{\rho ( \varepsilon ) }^*\T^n\setminus B^*\T^n}f_{\varphi_{H_{ \varepsilon }}}; \end{equation}
  but on $B_{\rho ( \varepsilon ) }^*\T^n\setminus B^*\T^n$, $f_{\varphi_{H_{ \varepsilon }}}$ is given by $\langle d_pH_{ \varepsilon }(q,p,t), p \rangle - H_{ \varepsilon }(q,p,t)$. It is thus enough to check that the integral of this quantity on  $(B_{\rho ( \varepsilon )}^*\T^n\setminus B^*\T^n)\times\T$ goes to zero as $ \varepsilon $ goes to zero, or else that $$ \varepsilon \Vert  \langle dH_{ \varepsilon }(q,p,t), p \rangle - H_{ \varepsilon }(q,p,t) \Vert_{C^0} \stackrel{\e\to 0}\longrightarrow 0.$$ 
Since $ \Vert H_{ \varepsilon } \Vert _{C^0}$ and $ \varepsilon \Vert d_pH_{ \varepsilon }(q,p,t) \Vert_{C^0}$ go to zero on $(B_{\rho(\varepsilon ) }^*\T^n\setminus B^*\T^n)\times\T$, this clearly holds.\\

We now compare the  Hofer distance with the Calabi invariant. 
\begin{Lem}\label{HoferCalabi}
Let $\varphi \in \Ham_S(B^*\T^n)$. Then,
$$
d^H_{\infty}(\id, \varphi) \geq \frac{1}{{\rm Vol}(B^*\T^n)}|\Cal(\varphi)|.
$$
\end{Lem}

\begin{proof}
Let $\varphi \in \Ham_S(B^*\T^n)$ and $H(q,p,t)$ be a (not necessarily convex) generating Hamiltonian. Recall from the definition of  $\Ham_S(B^*\T^n)$  that $H(q,p,t)=0$ on $\|p\|=1$
and that it admits a smooth extension $H_{ \varepsilon }: \rT^*\T^n\times \T \longrightarrow \R$ in $\cH_{ \varepsilon }(H)$, which is a function only of $t$ and $\|p\|^2$ outside $B^*\T^n\times \T$ and which is bounded by $ \varepsilon $ outside $B^*\T^n \times \T$ (see section \ref{subsecprop2}).

Now, denoting by $\varphi_{ \varepsilon }$ the time-one flow $\varphi^1_{H_{ \varepsilon }}$, we may write:

\beqano
d^H (\id, \varphi_{{ \varepsilon }}) &=&
\int_0^1 {\Osc}_{B_{\rho ( \varepsilon )}^*\T^n}(H_{\e,t})\,dt  \\
&= & \frac{1}{{\rm Vol}(B_{\rho ( \varepsilon )}^*\T^n)}  \int_0^1 \int_{\rT^*\T^n}  {\Osc}_{\rT^*\T^n}({H}_{\e,t})\, \omega^n\,dt  \; \\
&\geq & \frac{1}{{\rm Vol}(B_{\rho ( \varepsilon )}^*\T^n)} \left| \int_0^1 \int_{\rT^*\T^n}  {H}_{\e}(q,p,t)\, \omega^n\,dt \right|  \;\\
& =& \frac{1}{{\rm Vol}(B_{\rho ( \varepsilon )}^*\T^n)} \left| \Cal (\varphi_\e) \right| .
\eeqano

Hence, using (\ref{limitcalabi}), we obtain:
\beqano
d^H (\id , \varphi) &=& \lim_{ \varepsilon \to 0}d^H (\id, \varphi_{{ \varepsilon }}) \;  \\
&\geq& \lim_{ \varepsilon \to 0} \frac{1}{{\rm Vol}(B_{\rho ( \varepsilon )}^*\T^n)} \left| \Cal (\varphi_\e) \right|\; \\
&=&\frac{1}{{\rm Vol}(B^*\T^n)} \left| \Cal (\varphi)\right|.
\eeqano

Since $H(q,p,t+1)=H(q,p,t)$, we can conclude that:
\beqano
d^H_{\infty}(\id, \varphi) &\geq& \frac{1}{{\rm Vol}(B^*\T^n)} \lim_{k\rightarrow +\infty}\left( \frac{1}{k} \left| \Cal (\varphi_H^k ) \right|\right) \; \\
&=& \frac{1}{{\rm Vol}(B^*\T^n)}  \left| \int_0^1 \int_{B^*\T^n}  {H}(t,q,p)\, \omega^n\,dt \right| \; \\
&=& \frac{1}{{\rm Vol}(B^*\T^n)}|\Cal(\varphi)|.
\eeqano
\end{proof}

In order to find  our example of $\varphi \in \Ham_S(B^*\T^n)$  such that $\gamma_{\infty}(\id, \varphi) <  d^H_{\infty}(\id, \varphi)$, it is sufficient to find $\varphi$ such that 
$$
\gamma_{\infty}(\id, \varphi)  <  \frac{1}{{\rm Vol}(B^*\T^n)}|\Cal(\varphi)|.
$$
Let $U_{ \delta }=\big[\frac{1-\d}{2}, \frac{1+\d}{2}\big]^{n}$ be a cube of side $ \delta <\frac{1}{3}$ contained in $\T^n$. Let $H$ be a negative convex Hamiltonian of the form $H(q,p)=\g(q) ( \Vert p \Vert ^2-1)$, such that  $\g(q)\geq 0$ and

$$
\g(q) = \left\{ \begin{array}{lll} C && {\rm on}\; U_{\delta}\\
c && {\rm on}\; \T^n\setminus U_{2\delta}\,.
\end{array}\right.
$$
where $c<<C$.

Observe that:

$$
 \vert \Cal (\varphi_H^1) \vert :=\lim_{ \varepsilon \to 0} \vert \Cal (\varphi_{H_{ \varepsilon }}^1) \vert  \geq [\delta^n C+c (1-2^n\delta^n)] \int_{\{\|p|\leq 1\}}(1-\|p\|^2)\,dp\,.
$$
If we set $k:= \int_{\{\|p|\leq 1\}}(1-\|p\|^2)\,dp$, then
$$
|\Cal (\varphi_H^1)| \geq [\delta^n C+c (1- 2^n\delta^n)] k.
$$

In order to conclude the proof, it is sufficient to prove that 
$$ \lim_{ \varepsilon \to 0} \vert \overline{H}_{ \varepsilon }(p)\vert < \frac{\delta^n Ck}{{\rm Vol}(B^*\T^n)}  \quad \mbox{for all}\; p,$$
 where as usual $\overline{H}_{ \varepsilon }$ denotes the symplectic homogenization of $H_{ \varepsilon }$. Then, applying Proposition \ref{ex-corollary} and Lemma \ref{HoferCalabi}, it follows that
$$\gamma_{\infty}(\id, \varphi) <  \frac{1}{{\rm Vol}(B^*\T^n)}\vert \Cal(\varphi) \vert \leq d_{\infty}^H(\varphi , \id).$$

In order to prove this, we use the fact  that if there exists a Lagrangian submanifold $\L_p=\psi(\T^n\times \{p\})$, where $\psi \in \Ham(\rT^*\T^n)$, such that $-H_{ \varepsilon }\big|\L_p < A$, then $- \overline{H}_{ \varepsilon }(p)\leq A$; see Viterbo \cite[Theorem 3.2]{Viterbo2009}. We look for $\L_p$ in the  form
$$
\L_p= \{(q,p+df(q)): \; q\in\T^n \},
$$
with $f$ satisfying  the condition  $\|p+df(q)\| = 1$ on $U_{2\delta}$. 

\begin{Lem} 
For all vectors $p$ in $  {\mathbb R} ^n$ there exists a smooth  function $f$ on $\T^n$ such that $\Vert df(q)+p\Vert =1$ on $U_{ 2 \delta }$.
\end{Lem} 

\begin{proof} 
We must find a vector field $u(q)$ of norm $1$ on $U_{2 \delta }$, such that $df(q)=u(q)-p$ on $U_{ 2 \delta }$. Take $u$ to be constant on $U_{ 2 \delta }$, then $f(q)=\langle u-p,q\rangle$ and extend this to a smooth function on $\T^n$.
\end{proof} 
Then,
\beqano
 -H_{ \varepsilon }(q,p+df(q)) &=& 0 \quad {\rm on}\; U_{2\delta} \quad(\text{because}\;  \Vert df(q)+p \Vert =1) \\
- \varepsilon  \leq -H_{ \varepsilon  }(q,p+df(q)) &\leq& {c} \quad {\rm on}\; \T^n\setminus U_{2\delta}\,.
\eeqano

Hence $-\overline H (p) = |\overline H (p)| = \lim_{ \varepsilon \to 0} \vert {\overline H}_{ \varepsilon } (p)\vert  \leq {c}$ for all $p$. Therefore, we proved that $$\gamma_{\infty}(\varphi, \id)= \lim_{ \varepsilon \to 0} \Osc ({\overline H}_{ \varepsilon }) \leq   \lim_{ \varepsilon \to 0}\sup_{p\in {\mathbb R}^n}  \vert {\overline H}_{ \varepsilon }(p) \vert \leq c$$

Provided we choose our  constants to  satisfy ${c} <  \frac{\delta^n Ck}{{\rm Vol}(B^*\T^n)},$ this concludes the proof of Theorem  \ref{counterex}.
\end{proof}

 \begin{Rem} 
 The truncation of the Hamiltonian $H$ at level $0$ gives an example of a compact supported Hamiltonian map on the unit cotangent ball $B^*\T^{n}$ for which the Hofer and $\gamma$-distance differ. 
 \end{Rem}

%
%
%
%

\end{document}